\documentclass{article}
\usepackage[utf8]{inputenc}
\usepackage[T1]{fontenc}
\usepackage{amsmath,amssymb,amsthm,bbm,float,graphicx,geometry,lmodern,mathtools,parskip,setspace,subcaption,mathabx}
\usepackage{mathrsfs}
\usepackage{enumerate}
\usepackage{nicefrac}
\usepackage{todonotes}
\usepackage{comment}
\usepackage[colorlinks, linkcolor = blue!80!black, citecolor = blue!80!black, breaklinks, pdfauthor={Christian Moench}]{hyperref}
\usepackage{orcidlink}
\usepackage{titling}
\usepackage{fullpage}

\usepackage{tikz}
\usetikzlibrary{backgrounds}
\usetikzlibrary{patterns}
\usetikzlibrary{positioning, shapes.geometric}

\usepackage{caption}
\usepackage{graphicx}
\graphicspath{{Images/}}
\allowdisplaybreaks

\newcommand{\N}{\mathbb{N}}

\newtheorem{theorem}{Theorem}
\newtheorem{lemma}[theorem]{Lemma}

\theoremstyle{definition}

\newtheorem{remark}[theorem]{Remark}

\usepackage{nameref}

\makeatletter
\let\orgdescriptionlabel\descriptionlabel
\renewcommand*{\descriptionlabel}[1]{%
  \let\orglabel\label
  \let\label\@gobble
  \phantomsection
  \edef\@currentlabel{#1}%
  \let\label\orglabel
  \orgdescriptionlabel{#1}%
}

\renewcommand{\P}{\mathbb{P}}

\newcommand{\1}{\mathbbm{1}}

\newcommand{\E}{\mathbb{E}}

\newcommand{\R}{\mathbb{R}}

\newcommand{\Z}{\mathbb{Z}}

\pretitle{\centering\LARGE\scshape}
 \posttitle{\vskip 0.75cm}

 \predate{\vskip 0.5 cm \centering\large}
 \postdate{\par}

\usepackage[style = numeric, 
url = false,
abbreviate=false,
maxbibnames=9, sortcites=true,
doi = true,
backend = biber,
giveninits = true,
isbn=false]{biblatex}
\renewbibmacro{in:}{\ifentrytype{article}{}{\printtext{\bibstring{in}\intitlepunct}}}
\title{Inhomogeneous long-range percolation in the strong decay regime: recurrence in one dimension}

\thanksmarkseries{arabic}

\author{
Christian M\"{o}nch \orcidlink{0000-0002-6531-6482}\thanks{Johannes Gutenberg-Universität Mainz, Staudingerweg 9, 55128 Mainz, Germany} \\ cmoench@uni-mainz.de
}

\date{\today}

\begin{document}
\maketitle

\begin{spacing}{0.9}
\begin{abstract} 
\noindent We provide a sufficient criterion for the recurrence of spatial random graphs on the real line based on the scarceness of long-edges. In particular, this complements earlier recurrence results obtained by Gracar et al.\ [Electron. J. Probab. 27 (2022), article no. 60] and the transience criterion derived by the author in [Probab. Theory Related Fields \emph{189, no.\ 3-4} (2024), pp.\ 1129–1160.]

\smallskip
\noindent\footnotesize{{\textbf{AMS-MSC 2020}: 05C80, 60K35}

}
\end{abstract}
\end{spacing}
\section*{Introduction and background}
This note continues a line of investigation begun in \cite{gracar_recurrence_2022}, with the aim of characterising the recurrence and transience regimes of simple random walk on the infinite clusters of inhomogeneous long-range percolation models in Euclidean space. For the sake of exposition, consider the integer lattice $\Z^d$ and connect pairs of points $x,y\in \Z^d$ independently by an edge with probability
\[
\phi(|x-y|),\quad xy\in(\Z^d)^{[2]}.
\]
Here and throughout, the notation $A^{[2]}=\{xy: x\in A, y\in A, x\neq y\}$ refers to unordered pairs of elements of a set $A$. The function $\phi:\N\to [0,1]$ is called the \emph{connection function} and we call the random geometric graph ensemble generated by the above procedure \emph{homogeneous long-range percolation} on $\Z^d$. If there exists an infinite connected component, the \emph{infinite cluster}, in this graph, then it is necessarily unique and contains a positive fraction of all vertices. One way to see this is to apply the general argument of Burton and Keane \cite{burton_density_1989}. If it exists, we say that the infinite cluster $K=(V(K),E(K))$ is \emph{recurrent}, if conditionally on $0\in V(K)$, the simple random walk on $K$ initiated at $0$ returns to $0$ almost surely. A cluster that is not recurrent is \emph{transient}. Note that the ergodic theorem implies that recurrence is a $0-1$-property of the infinite cluster.

Schulman initiated the study of long-range percolation in 1983 for the choice $d=1$ and $\phi(r)= p r^{-\delta}$ where $p\in(0,1)$ \cite{schulman_long_1983}. He showed that there exists no infinite cluster if either $\delta>2$ or $\delta=2$ and $p$ is sufficiently small. Following these early results, research on long-range percolation continued actively throughout the eighties, see e.g.\ \cite{aizenman_discontinuity_1986,newman_one-dimensional_1986,newman_critical_1986,aizenman_sharpness_1987,aizenman_uniqueness_1987,imbrie_intermediate_1988}. A remarkable mathematical feature of the model that was obtained in \cite{aizenman_discontinuity_1986} is the presence of a discontinuous phase transition in scale-invariant long-range percolation, i.e.\ in the above setting with $\delta=2$.

Amidst renewed efforts to investigate the properties of the model in the early 2000's, illustrated for instance by the works \cite{benjamini_diameter_2001,coppersmith_diameter_2002,biskup_scaling_2004,berger_lower_2004}, Berger published an influential paper \cite{berger_transience_2002} in which he established a variety of results about the infinite cluster in long-range percolation. Of particular relevance for the present work is his proof of the following fact: when $d\in\{1,2\}$ and $\phi(r)\asymp r^{-\delta d}$, the infinite cluster is recurrent if and only if $\delta\geq 2$. Moreover, if $d=2$ and $\delta>2$, then Berger's arguments yield recurrence without \emph{any} assumptions on the edge dependencies. Here, one should bear in mind that if $d=1$ and $\delta>2$, then it is necessary to allow $\phi(r)=1$ for some $r$, otherwise there is no infinite cluster and the recurrence statement becomes trivial, by Schulman's result. In recent years, the number of works concerned with long-range percolation has steadily grown. Some recent exciting results have been obtained in the papers \cite{biskup_quenched_2021,hutchcroft_sharp_2022,ding2023uniquenesscriticallongrangepercolation,baumler_isoperimetric_2024,jorritsma_cluster-size_2024}. On the subject of recurrence vs.\ transience of the infinite cluster, B\"aumler \cite{baumler_recurrence_2023} provided a unified argument to establish recurrence without edge independence assumptions whenever $d\in\{1,2\}$ and $\phi(r)\asymp r^{-\delta d}$ with $\delta \leq 2$ and thereby gave an alternative and more general proof of Berger's recurrence results.

\emph{Inhomogenous} long-range percolation models are a related subject of intense current research, see e.g.\ \cite{gracar_age-dependent_2019,gracar_percolation_2021,gracar_chemical_2022,gracar_emergence_2023,gracar2023finitenesspercolationthresholdinhomogeneous,jorritsma2024clustersizedecaysupercriticalkernelbased,jorritsma2024largedeviationsgiantsupercritical}. Of particular relevance are models that incorporate random geometric graphs with power law degree distributions, so-called \emph{spatial scale-free networks}. Spatial scale-free networks combine desirable features of complex network models and of classical Euclidean percolation models. In the most common \emph{kernel-based} variant of inhomogeneous long-range percolation, the vertices are first equipped with i.i.d.\ $\operatorname{Uniform}(0,1)$ marks $(U_x)_{x\in\Z^d}$ and the connection function takes the form
\begin{equation}\label{eq:edgesinhom}
\phi(g(U_x,U_y)|x-y|),\quad xy\in(\Z^d)^{[2]},
\end{equation}
where $g:(0,1)\to [0,\infty]$ is called the \emph{kernel} and $\phi$ is now a function of arguments in $[0,\infty]$. Connections are then sampled independently given the marks, i.e.\ the edges in the resulting random graph are only independent if $g$ is constant, in which case the model coincides with homogeneous long-range percolation. The earliest instances of this model were introduced independently by Krioukov et al.\ \cite{krioukov_hyperbolic_2010} and Deijfen and van der Hofstad \cite{deijfen_scale-free_2013}, corresponding essentially to choosing the kernel of product form $g(s,t)=h(s)h(t)$ for a suitable functions $h$. The general variant of the models which mainly inspired the present work appeared first in \cite{bringmann_geometric_2019} in a finite setting, and the importance of considering kernels that are qualitatively different from product kernels was first recognised in \cite{gracar_recurrence_2022}, where it was observed that by choosing different kernel functions one can unify the study of many spatial scale-free network models that had previously appeared in the literature under different labels. Different choices of such specific kernels $g$ and their effect on the graph topology are discussed for instance in \cite{gracar_recurrence_2022,gracar2023finitenesspercolationthresholdinhomogeneous} and \cite{jorritsma2024clustersizedecaysupercriticalkernelbased}.

The works \cite{gracar2023finitenesspercolationthresholdinhomogeneous,moench_2024a} further introduce three different implicit parameter regimes for inhomogenous long-range models. To characterise these regimes, we define the \emph{effective decay exponent} via
\[
\delta_{\mathsf{eff}}=\lim_{n\to\infty}\frac{\log \big(\int_{1/n}^1\int_{1/n}^1 \phi(g(s,t) n )\,\textup{d}s\textup{d}t\big)}{\log n}.
\]
If well-defined, the quantity $\delta_{\mathsf{eff}}$ plays a similar role to the exponent $\delta$ in homogeneous long-range percolation. If $\delta_{\mathsf{eff}}>2$, the inhomogeneous model is said to be in the \emph{strong decay regime}. The weak decay regime is characterised by $\delta_{\mathsf{eff}}<2$, and it was shown in \cite{moench_2024a} that weak decay typically results in transient infinite clusters, no matter the dimension. The boundary case $\delta_{\mathsf{eff}}=2$ is dubbed the pseudo-scale-invariant regime in analogy with scale-invariant long-range percolation. Not very much is known about this regime for truly inhomogeneous models with non-product kernels, see \cite{gracar2023finitenesspercolationthresholdinhomogeneous} for a more detailed discussion.

Below, we present a simple proof for recurrence of $1$-dimensional long-range percolation models in the strong decay regime. For the independent homogeneous case, our argument is significantly shorter than Berger's reduction to the `continuum bond model' used in \cite{berger_transience_2002}, but also weaker -- it does not cover the scale-invariant case. Since we do not need independence or the connection to random walks with long-range jumps employed by B\"aumler in \cite{baumler_recurrence_2023}, our approach is more sensitive to edge correlations then previous methods and therefore does yield improved results for models that produce graphs with heavy-tailed degree distributions.

\section*{Results and proofs}
We begin by formulating the mathematical framework rigorously. A (one-dimensional) \emph{random geometric graph} $G=(V,E)$ is defined as a countable, possibly random, set $V\subset\R$ of vertices, together with a random relation $E\subset V^{[2]}$ indicating edges. Typically, the presence of edges is in some way adapted to the geometry of the real line. This distinguishes a random geometric graph from a plain random graph. In the examples above, the adaptation is a consequence of the edge probabilities depending on the Euclidean distance of the endpoints. By $\Z_{\mathsf{nn}}$ we denote the usual nearest-neighbour graph on the vertex set $\Z$. 

The question of recurrence and transience can be most conveniently discussed in the setting of electrical networks, which has the immediate advantage that there is only one source of randomness to consider, namely the underlying graph, instead of the random graph and the random walk on the random graph. We assume familiarity with the basic operations on and relations between electrical networks, comprehensive introductions to the subject are provided for instance in the monographs \cite{doyle_random_1984} and \cite{lyons_probability_2016}. A random geometric graph $G=(V,E)$ endowed with a map $C:E\to[0,\infty]$ of \emph{conductances} is called an \emph{electrical network}. We write such networks as triples $G=(V,E,C)$. If no conductances are explicitly specified, we always tacitly assume $C\equiv 1$. We allow $C_e=\infty$ for an edge $e=xy\in E$ and identify the resulting network with the network where $x$ and $y$ are the same vertex with no self-loop and all other conductances unchanged. The dynamics of the random walk $X^C$ on $G$ with finite conductances are given by
\[
\P(X^C_{t+1}=y|X^C_t=x)=\frac{C_{xy}}{\sum_{xz\in E}C_{xz}}.
\]
If $X^C$ is not well-defined, i.e.\ if there exists some $x\in V$ with $\sum_{xz\in E}C_{xz}=\infty$, then we declare the cluster of $x$ in $G$ transient by convention. In particular, we treat clusters in graphs that are not locally finite as transient. Otherwise, we define, for $x,z\in V$,
\[
\mathsf{C}(x\leftrightarrow z)=\Big(\sum_{xy\in E}C_{xy}\Big)\P\big(X^C \text{ visits }z \text{ before returning to }x\big|X^C_0=x\big),
\]
the \emph{effective conductance} between $x$ and $z$. The notion extends to sets. For instance $\mathsf{C}(A\leftrightarrow B)$ for $A,B\subset V$ denotes the effective conductance in the new network obtained by identifying all vertices in $A$ with a single vertex $a$ and all vertices in $B$ with a single vertex $b$. The limit
\[
\mathsf{C}(x\leftrightarrow \infty) = \lim_{n\to\infty} \mathsf{C}\big(x\leftrightarrow (\R\setminus [-n,n])\cap \{y\in V: x\text{ is connected to }y \text{ in }G\}\big)
\]
always exists and the cluster of $x$ is recurrent (for $X^C$) precisely if $\mathsf{C}(x\leftrightarrow \infty)=0$.

The geometry of the line graph $\Z_{\mathsf{nn}}$ is so restrictive that recurrence holds for practically any translation invariant assignment of conductances.

\begin{theorem}\label{thm:neaarest-neighbour}
Consider $\Z_{\mathsf{nn}}$ with conductances $(C_{z,z+1})_{z\in\Z}$ such that
\begin{enumerate}[(a)]
    \item the field $C=(C_{z,z+1})_{z\in\Z}$ is translation invariant and ergodic with respect to $(\Z,+)$,
    \item $C_{0,1}$ is non-trivial in the sense that \ $\P(C_{0,1}<\infty)>0$.
\end{enumerate}
Then, 
\begin{equation}\label{thm:ceffdecay}
\E\big[\mathsf{C}(0\leftrightarrow \Z\setminus\{-n,-n+1,\dots,n-1,n\} ) \big] = O(1/n),
\end{equation}
and consequently the electrical network is recurrent.
\end{theorem}
\begin{remark}
Berger \cite{berger_transience_2002} gives two contrasting sufficient conditions for recurrence: firstly, \cite[Theorem~1.9]{berger_transience_2002} provides recurrence for any network built from a recurrent graph of bounded degree, but under the assumption that the conductances are i.i.d.\ random variables in $(0,\infty)$. Secondly, \cite[Theorem~3.9]{berger_transience_2002} establishes recurrence in the delicate case of $\Z^2$ for any translation invariant field of conductances with Cauchy tail decay, which is essentially sharp.  
\end{remark}

\begin{proof}[Proof of Theorem~\ref{thm:neaarest-neighbour}]
By non-triviality of $C_{0,1}$, there exists $M<\infty$ such that
\[
\P(C_{0,1}\leq M)=:\varepsilon>0.
\]
For $n\in\N$ let $e(n)$ denote the edge between $n$ and $n+1$ and let $e(-n)$ denote the edge between $-n$ and $-n+1$. Set $C_{z}=C_{e(z)}$ for $z\in \Z\setminus\{0\}$ for ease of notation. We say that an edge $e(z)$  is \emph{bad}, if $C_{z}\leq M$. Identifying all vertices in $\Z\setminus\{-n,-n+1,\dots,n-1,n\}$, we see that we can replace the network by a pair of vertices with precisely one edge of counductance $C'$ without changing the effective conductance $$\mathsf{C}(0\leftrightarrow \Z\setminus\{-n,-n+1,\dots,n-1,n\},$$ if we set 
\[C'=C'_++C'_-,\]
where 
\[
C'_+=\Big(\sum_{z=1}^{n}\frac{1}{C_z}\Big)^{-1},\quad C'_-=\Big(\sum_{z=-n}^{-1}\frac{1}{C_z}\Big)^{-1}.
\]
The effective conductance in this new network coincides trivially with $C'$ and it suffices to show that $\E[C']=O(1/n)$. Since $C_z$ may be replaced by $C_{-z}$ for any $z$ without affecting the assumptions of the theorem, we may even limit ourselves to showing that $\E[C'_+]=O(1/n)$. We obtain
\begin{align*}
\E\big[C'_+\big]& = \E\Big[\Big(\sum_{z\in \Z_+}\frac{1}{C_z}\Big)^{-1} \Big]\\
& = \int_0^\infty \E \exp\Big[-\lambda \sum_{z\in \Z_+}\frac{1}{C_z} \Big]\,\textup{d} \lambda\\
& \leq \int_0^\infty \E \exp\Big(-\lambda \sum_{j=1}^n \1\{e(j) \text{ is bad}\}\frac{1}{C_{j}} \Big)\,\textup{d} \lambda\\
& \leq \int_0^\infty \E \exp\Big(-\frac{\lambda}{M} \sum_{j=1}^n \1\{e(j) \text{ is bad}\} \Big)\,\textup{d} \lambda.
\end{align*}
By Birkhoff's ergodic theorem,
\[
\lim_{n\to\infty}\frac{1}{n} \sum_{j=1}^n \1\{e(j) \text{ is bad}\}=\varepsilon
\]
almost surely, and hence
\[\sum_{j=1}^n \1\{e(j) \text{ is bad} \}\geq \varepsilon n/2\]
for all sufficiently large $n$. Consequently, we obtain
\[
\E\big[C'_+\big]\leq \int_0^\infty \exp\Big(-\frac{\lambda \varepsilon}{2M} n \Big)\,\textup{d} \lambda =O(1/n),
\]
as $n\to\infty$.
\end{proof}
Theorem~\ref{thm:neaarest-neighbour} provides a sufficient criterion for recurrence of electrical networks built from $\Z_{\mathsf{nn}}$. Our next goal is to derive a simple recurrence criterion for long-range models from this. To this end, consider any connected electric network $G=(V,E,C)$ and construct a new network $G_{\mathsf{nn}}$ on $\Z_{\mathsf{nn}}$ from $G$ by performing the following steps:
\begin{itemize}
    \item for all $z\in \Z$, merge the vertices in $V\cap [z,z+1)$ into a single vertex located at $z$ (which we identify with the point $z$), whilst keeping all edges and their conductances;
    \item remove any self-loops;
    \item replace each edge $e$ of original length $\ell>0$ by a nearest neighbour path of $\lceil\ell\rceil$ edges, with conductance $\lceil\ell\rceil C_e$ on each of the edges of the new path;
    \item where necessary, shorten the paths thus obtained such that their length corresponds to the distance of their (new) endpoints;
    \item project the paths onto $\Z_{\mathsf{nn}}$ and add the conductances onto the edge conductances of the corresponding nearest-neighbour edges.
\end{itemize}
It follows from this construction and the monotonicity of effective conductance in the network operations involved, that if $G_{\mathsf{nn}}$ is recurrent, so is $G$, and we obtain the following recurrence criterion from Theorem~\ref{thm:neaarest-neighbour}.
\begin{lemma}\label{lem:Xcount}
Suppose that $G=(V(G),E(G))$ is a random geometric graph with $V(G)\subset \R$ which is distributionally translation invariant and ergodic with respect to the action of a subgroup of $(\R,+)$. Further define the random variable
\[
\mathsf{X}=\sum_{x\in V: x\leq 0}\sum_{y\in V: y>0}\1\{xy\in E\}.
\]
We have that all clusters of $G$ are recurrent, if $\P(\mathsf{X}<\infty)>0$.
\end{lemma}
We call $X$ the \emph{number of edges above }$0+$.

Next, we formulate our recurrence result for the weight-dependent random connection model over $\Z_{\mathsf{nn}}$, i.e.\ the random geometric graph constructed by adding the edges produced via \eqref{eq:edgesinhom} to $\Z_{\mathsf{nn}}$.
\begin{theorem}\label{thm:WDRCMrecurr}
The infinite cluster of the weight-dependent random connection model on $\Z_{\mathsf{nn}}$ in the strong decay regime is recurrent in dimension $1$.    
\end{theorem}
Note that finiteness of $\mathsf{X}$ for the weight-dependent random connection model in the strong decay regime is in fact already implicit in \cite{gracar2023finitenesspercolationthresholdinhomogeneous}, so we could apply Lemma~\ref{lem:Xcount} directly. Instead, we set up the proof of Theorem~\ref{thm:WDRCMrecurr} in a slightly stronger form which is independent of the existence of $\delta_{\mathsf{eff}}$. The argument is based on the insights of \cite{jacob2024longedgeserasesubcritical}. We say that \emph{long edges eventually disappear at the exponential scale} in a random geometric graph $G=(V,E)$, if 
\begin{equation}\label{eq:explongedge}
\sum_{k\in\N}\P\big(\exists x\in V(G)\cap [-2^k,2^k], y\in V(G): |x-y|\geq 2^k, xy\in E(G) \big)<\infty.
\end{equation}
\begin{lemma}\label{lem:Recurrencecrit}
Suppose that $G=(V(G),E(G))$ is random geometric graph with $V(G)\subset \R$ that satisfies 
\begin{enumerate}[(a)]
    \item $G$ is distributionally translation invariant and ergodic with respect to the action of a subgroup of $(\R,+)$;
    \item long edges in $G$ eventually disappear at the exponential scale.
\end{enumerate}
Then all clusters of $G$ are recurrent almost surely.
\end{lemma}
\begin{proof}[Proof of Lemma~\ref{lem:Recurrencecrit}]
It suffices to show that the assumption of Lemma~\ref{lem:Xcount} are satisfied. Let $X_k$ denote the number of edges above $0+$ that emanate from $[2^k,2^{k+1})$. By symmetry, $\mathsf{X}$ is finite a.s.\ if and only if $\sum_{k\in\N}X_k$ is finite. The latter statement follows from \eqref{eq:explongedge} and the first Borel-Cantelli lemma.
\end{proof}
\begin{proof}[Proof of Theorem~\ref{thm:WDRCMrecurr}]
It is shown in \cite{jahnel2023existencesubcriticalpercolationphases,jacob2024longedgeserasesubcritical} that the weight-dependent random connection model with $\delta_{\mathsf{eff}}>2$ satisfies
\[
\P\big(\exists x\in V(G)\cap [-n,n], y\in V(G): |x-y|\geq n, xy\in E(G) \big)=O(n^{-\vartheta}),
\]
for some $\vartheta>0$, which implies \eqref{eq:explongedge}.
\end{proof}

\section*{Concluding remarks and future work}
We have provided a straightforward sufficient condition in Lemma~\ref{lem:Recurrencecrit} for the recurrence of a long-range percolation model in one dimension. Lemma~\ref{lem:Xcount} shows that this is, in principle, an electric network version of Schulman's non-existence result for infinite clusters: in classical long-range percolation, the number of edges above $0+$ is a sum of independent \emph{non-trivial} Bernoullis, thus $\mathsf{X}$ is finite with positive probability if and only if $\mathsf{X}$ vanishes with positive probability. The latter is Schulman's sufficient criterion for the non-existence of an infinite cluster.

For the sake of brevity, we have only applied our result to the weight-dependent random connection model over $\Z$, whereas the original model is defined on a Poisson process. Also in that case, if edges are added in a translation invariant manner to create an infinite cluster, our technique can be applied.

Homogeneous long-range percolation with $\delta=2$ displays a non-trivial percolation phase transition. However, if there is an infinite cluster, it must be recurrent. This shows that our approach cannot work in the pseudo-scale-invariant regime. Nevertheless, the above result improves upon Bäumler's \cite{baumler_recurrence_2023} and Berger's \cite{berger_transience_2002} results if the percolation graph has a degree distribution without second moments whilst still satisfying \eqref{eq:explongedge}. Several models in parameter regimes that are interesting for the modelling of network phenomena have these features. For example, our result complements those of \cite{baumler_recurrence_2023} and \cite{moench_2024a} in the case of the soft-Boolean model in one dimension and settles the recurrence/transience question for this instance of inhomogeneous long-range percolation up to the boundary case $\delta_{\mathsf{eff}}=2$, cf.\ \cite[p.\ 6]{baumler_recurrence_2023}.

To derive criteria as sharp as possible for recurrence/transience and existence/non-existence of an infinite cluster in the pseudo-scale-invariant regime in one dimension remains an interesting task for future work. Also, the extension of the present results to the more difficult case of recurrence in dimension $2$ under additional assumptions on the edge correlations seems to be within reach, but requires a far more careful analysis of the corresponding electric network.

\noindent\textbf{\large Acknowledgement.} I would like to express my gratitude to Lukas Lüchtrath for insightful discussions about the subject of this work and for spotting several mistakes in the first draft of the manuscript.

\noindent\textbf{\large Funding acknowledgement.} The author's research is funded by Deutsche Forschungsgemeinschaft (DFG, German Research Foundation) – grant no. 443916008 (SPP 2265).

\section*{References}
\renewcommand*{\bibfont}{\footnotesize}
\printbibliography[heading = none]

@article {moench_2024a,
    AUTHOR = {M\"onch, Christian},
     TITLE = {Inhomogeneous long-range percolation in the weak decay regime},
   JOURNAL = {Probab. Theory Related Fields},
  FJOURNAL = {Probability Theory and Related Fields},
    VOLUME = {189},
      YEAR = {2024},
    NUMBER = {3-4},
     PAGES = {1129--1160},
      ISSN = {0178-8051,1432-2064},
   MRCLASS = {05C80 (60K35)},
  MRNUMBER = {4771112},
       DOI = {10.1007/s00440-024-01281-5},
       URL = {https://doi.org/10.1007/s00440-024-01281-5},
}

@misc{jahnel2023existencesubcriticalpercolationphases,
      title={Existence of subcritical percolation phases for generalised weight-dependent random connection models}, 
      author={Benedikt Jahnel and Lukas Lüchtrath},
      year={2023},
      eprint={2302.05396},
      archivePrefix={arXiv},
      primaryClass={math.PR},
      url={https://arxiv.org/abs/2302.05396}, 
}

@inproceedings{gracar_emergence_2023,
	location = {Cham},
	title = {The Emergence of a Giant Component in One-Dimensional Inhomogeneous Networks with Long-Range Effects},
	isbn = {978-3-031-32296-9},
	doi = {10.1007/978-3-031-32296-9_2},
	pages = {19--35},
	booktitle = {Algorithms and Models for the Web Graph},
	publisher = {Springer Nature Switzerland},
	author = {Gracar, Peter and Lüchtrath, Lukas and Mönch, Christian},
	editor = {Dewar, Megan and Prałat, Paweł and Szufel, Przemysław and Théberge, François and Wrzosek, Małgorzata},
	date = {2023},
	langid = {english},
}

@misc{jorritsma2024clustersizedecaysupercriticalkernelbased,
      title={Cluster-size decay in supercritical kernel-based spatial random graphs}, 
      author={Joost Jorritsma and Júlia Komjáthy and Dieter Mitsche},
      year={2024},
      eprint={2303.00724},
      archivePrefix={arXiv},
      primaryClass={math.PR},
      url={https://arxiv.org/abs/2303.00724}, 
}

@misc{jorritsma2024largedeviationsgiantsupercritical,
      title={Large deviations of the giant in supercritical kernel-based spatial random graphs}, 
      author={Joost Jorritsma and Júlia Komjáthy and Dieter Mitsche},
      year={2024},
      eprint={2404.02984},
      archivePrefix={arXiv},
      primaryClass={math.PR},
      url={https://arxiv.org/abs/2404.02984}, 
}

@misc{gracar2023finitenesspercolationthresholdinhomogeneous,
      title={Finiteness of the percolation threshold for inhomogeneous long-range models in one dimension}, 
      author={Peter Gracar and Lukas Lüchtrath and Christian Mönch},
      year={2023},
      eprint={2203.11966},
      archivePrefix={arXiv},
      primaryClass={math.PR},
      url={https://arxiv.org/abs/2203.11966}, 
}

@misc{ding2023uniquenesscriticallongrangepercolation,
      title={Uniqueness of the critical long-range percolation metrics}, 
      author={Jian Ding and Zherui Fan and Lu-Jing Huang},
      year={2023},
      eprint={2308.00621},
      archivePrefix={arXiv},
      primaryClass={math.PR},
      url={https://arxiv.org/abs/2308.00621}, 
}

@article{jorritsma_cluster-size_2024,
	title = {Cluster-size decay in supercritical long-range percolation},
	volume = {29},
	issn = {1083-6489, 1083-6489},
	doi = {10.1214/24-EJP1135},
	pages = {Paper No. 82, 36 pp.},
	journaltitle = {Electronic Journal of Probability},
    shortjournal = {Electron. J. Probab.},
	author = {Jorritsma, Joost and Komjáthy, Júlia and Mitsche, Dieter},
	urldate = {2024-08-09},
	date = {2024-01},
}

@article{gracar_age-dependent_2019,
	title = {The age-dependent random connection model},
	volume = {93},
	issn = {0257-0130},
	url = {https://mathscinet.ams.org/mathscinet-getitem?mr=4032928},
	doi = {10.1007/s11134-019-09625-y},
	pages = {309--331},
	number = {3},
	journaltitle = {Queueing Systems. Theory and Applications},
	shortjournal = {Queueing Syst.},
	author = {Gracar, Peter and Grauer, Arne and Lüchtrath, Lukas and Mörters, Peter},
	urldate = {2021-07-07},
	date = {2019},
	mrnumber = {4032928},
}

@article{aizenman_discontinuity_1986,
	title = {Discontinuity of the percolation density in one-dimensional $1/|x- y|^2$ percolation models},
	volume = {107},
	issn = {0010-3616,1432-0916},
	pages = {611--647},
	number = {4},
	journaltitle = {Communications in Mathematical Physics},
	shortjournal = {Comm. Math. Phys.},
	author = {Aizenman, M. and Newman, C. M.},
	date = {1986},
	mrnumber = {868738},
}

@article{bringmann_geometric_2019,
	title = {Geometric inhomogeneous random graphs},
	volume = {760},
	issn = {0304-3975,1879-2294},
	doi = {10.1016/j.tcs.2018.08.014},
	pages = {35--54},
	journaltitle = {Theoretical Computer Science},
	shortjournal = {Theoret. Comput. Sci.},
	author = {Bringmann, Karl and Keusch, Ralph and Lengler, Johannes},
	date = {2019},
	mrnumber = {3913223},
}

@article{burton_density_1989,
	title = {Density and uniqueness in percolation},
	volume = {121},
	issn = {0010-3616,1432-0916},
	pages = {501--505},
	number = {3},
	journaltitle = {Communications in Mathematical Physics},
	shortjournal = {Comm. Math. Phys.},
	author = {Burton, R. M. and Keane, M.},
	date = {1989},
	mrnumber = {990777},
}

@article{deijfen_scale-free_2013,
	title = {Scale-free percolation},
	volume = {49},
	issn = {0246-0203,1778-7017},
	doi = {10.1214/12-AIHP480},
	pages = {817--838},
	number = {3},
	journaltitle = {Annales de l'Institut Henri Poincaré Probabilités et Statistiques},
	shortjournal = {Ann. Inst. Henri Poincaré Probab. Stat.},
	author = {Deijfen, Maria and van der Hofstad, Remco and Hooghiemstra, Gerard},
	date = {2013},
	mrnumber = {3112435},
}

@article{gracar_chemical_2022,
	title = {Chemical distance in geometric random graphs with long edges and scale-free degree distribution},
	volume = {395},
	issn = {0010-3616,1432-0916},
	doi = {10.1007/s00220-022-04445-3},
	pages = {859--906},
	number = {2},
	journaltitle = {Communications in Mathematical Physics},
	shortjournal = {Comm. Math. Phys.},
	author = {Gracar, Peter and Grauer, Arne and Mörters, Peter},
	date = {2022},
	mrnumber = {4487527},
}

@article{gracar_recurrence_2022,
	title = {Recurrence versus transience for weight-dependent random connection models},
	volume = {27},
	issn = {1083-6489},
	doi = {10.1214/22-ejp748},
	pages = {Paper No. 60, 31 pp.},
	journaltitle = {Electronic Journal of Probability},
	shortjournal = {Electron. J. Probab.},
	author = {Gracar, Peter and Heydenreich, Markus and Mönch, Christian and Mörters, Peter},
	date = {2022},
	mrnumber = {4417198},
}

@article{gracar_percolation_2021,
	title = {Percolation phase transition in weight-dependent random connection models},
	volume = {53},
	issn = {0001-8678,1475-6064},
	doi = {10.1017/apr.2021.13},
	pages = {1090--1114},
	number = {4},
	journaltitle = {Advances in Applied Probability},
	shortjournal = {Adv. in Appl. Probab.},
	author = {Gracar, Peter and Lüchtrath, Lukas and Mörters, Peter},
	date = {2021},
	mrnumber = {4342578},
}

@article{newman_one-dimensional_1986,
	title = {One-dimensional $1/|j-i|^s$ percolation models: the existence of a transition for $s \leq 2$},
	volume = {104},
	issn = {0010-3616,1432-0916},
	shorttitle = {One-dimensional 1/{\textbar}j-i{\textbar}{\textasciicircum}s percolation models},
	pages = {547--571},
	number = {4},
	journaltitle = {Communications in Mathematical Physics},
	shortjournal = {Comm. Math. Phys.},
	author = {Newman, C. M. and Schulman, L. S.},
	date = {1986},
	mrnumber = {841669},
}

@article{schulman_long_1983,
	title = {Long range percolation in one dimension},
	volume = {16},
	issn = {0305-4470,1751-8121},
	pages = {L639--L641},
	number = {17},
	journaltitle = {Journal of Physics. A. Mathematical and General},
	shortjournal = {J. Phys. A},
	author = {Schulman, L. S.},
	date = {1983},
	mrnumber = {723249},
}

@article{krioukov_hyperbolic_2010,
	title = {Hyperbolic geometry of complex networks},
	volume = {82},
	issn = {1539-3755,1550-2376},
	doi = {10.1103/PhysRevE.82.036106},
	pages = {036106, 18},
	number = {3},
	journaltitle = {Physical Review E. Statistical, Nonlinear, and Soft Matter Physics},
	shortjournal = {Phys. Rev. E (3)},
	author = {Krioukov, Dmitri and Papadopoulos, Fragkiskos and Kitsak, Maksim and Vahdat, Amin and Boguñá, Marián},
	date = {2010},
	mrnumber = {2787998},
}

@article{aizenman_sharpness_1987,
	title = {Sharpness of the phase transition in percolation models},
	volume = {108},
	issn = {0010-3616,1432-0916},
	pages = {489--526},
	number = {3},
	journaltitle = {Communications in Mathematical Physics},
	shortjournal = {Comm. Math. Phys.},
	author = {Aizenman, Michael and Barsky, David J.},
	date = {1987},
	mrnumber = {874906},
}

@article{baumler_recurrence_2023,
	title = {Recurrence and transience of symmetric random walks with long-range jumps},
	volume = {28},
	issn = {1083-6489},
	doi = {10.1214/23-ejp998},
	pages = {Paper No. 106, 24 pp.},
	journaltitle = {Electronic Journal of Probability},
	shortjournal = {Electron. J. Probab.},
	author = {Bäumler, Johannes},
	date = {2023},
	mrnumber = {4632146},
}

@book{doyle_random_1984,
	title = {Random walks and electric networks},
	volume = {22},
	isbn = {978-0-88385-024-4},
	series = {Carus Mathematical Monographs},
	pagetotal = {xiv+159},
	publisher = {Mathematical Association of America, Washington, {DC}},
	author = {Doyle, Peter G. and Snell, J. Laurie},
	date = {1984},
	mrnumber = {920811},
}

@book{lyons_probability_2016,
	title = {Probability on trees and networks},
	volume = {42},
	isbn = {978-1-107-16015-6},
	series = {Cambridge Series in Statistical and Probabilistic Mathematics},
	pagetotal = {xv+699},
	publisher = {Cambridge University Press, New York},
	author = {Lyons, Russell and Peres, Yuval},
	date = {2016},
	mrnumber = {3616205},
	doi = {10.1017/9781316672815},
}

@article{imbrie_intermediate_1988,
	title = {An intermediate phase with slow decay of correlations in one-dimensional $1/|x- y|^2$ percolation, Ising and Potts models},
	volume = {118},
	issn = {0010-3616,1432-0916},
	pages = {303--336},
	number = {2},
	journaltitle = {Communications in Mathematical Physics},
	shortjournal = {Comm. Math. Phys.},
	author = {Imbrie, J. Z. and Newman, C. M.},
	date = {1988},
	mrnumber = {956170},
}

@article{aizenman_uniqueness_1987,
	title = {Uniqueness of the infinite cluster and continuity of connectivity functions for short and long range percolation},
	volume = {111},
	issn = {0010-3616,1432-0916},
	pages = {505--531},
	number = {4},
	journaltitle = {Communications in Mathematical Physics},
	shortjournal = {Comm. Math. Phys.},
	author = {Aizenman, M. and Kesten, H. and Newman, C. M.},
	date = {1987},
	mrnumber = {901151},
}

@article{newman_critical_1986,
	title = {Some critical exponent inequalities for percolation},
	volume = {45},
	issn = {0022-4715,1572-9613},
	doi = {10.1007/BF01021076},
	pages = {359--368},
	number = {3},
	journaltitle = {Journal of Statistical Physics},
	shortjournal = {J. Statist. Phys.},
	author = {Newman, C. M.},
	date = {1986},
	mrnumber = {869320},
}

@article{baumler_isoperimetric_2024,
	title = {Isoperimetric lower bounds for critical exponents for long-range percolation},
	volume = {60},
	issn = {0246-0203,1778-7017},
	doi = {10.1214/22-aihp1342},
	pages = {721--730},
	number = {1},
	journaltitle = {Annales de l'Institut Henri Poincaré Probabilités et Statistiques},
	shortjournal = {Ann. Inst. Henri Poincaré Probab. Stat.},
	author = {Bäumler, Johannes and Berger, Noam},
	date = {2024},
	mrnumber = {4718396},
}

@article{benjamini_diameter_2001,
	title = {The diameter of long-range percolation clusters on finite cycles},
	volume = {19},
	issn = {1042-9832,1098-2418},
	doi = {10.1002/rsa.1022},
	pages = {102--111},
	number = {2},
	journaltitle = {Random Structures \& Algorithms},
	shortjournal = {Random Structures Algorithms},
	author = {Benjamini, Itai and Berger, Noam},
	date = {2001},
	mrnumber = {1848786},
}

@article{berger_transience_2002,
	title = {Transience, recurrence and critical behavior for long-range percolation},
	volume = {226},
	issn = {0010-3616,1432-0916},
	url = {https://arxiv.org/abs/math/0110296v3},
	doi = {10.1007/s002200200617},
	pages = {531--558},
	number = {3},
	journaltitle = {Communications in Mathematical Physics},
	shortjournal = {Comm. Math. Phys.},
	author = {Berger, Noam},
	date = {2002},
	mrnumber = {1896880},
	addendum = {Corrected version: {arXiv}:math/0110296v3},
}

@misc{berger_lower_2004,
      title={A lower bound for the chemical distance in sparse long-range percolation models}, 
      author={Noam Berger},
      year={2004},
      eprint={math/0409021},
      archivePrefix={arXiv},
      primaryClass={math.PR},
      url={https://arxiv.org/abs/math/0409021}, 
}

@article{biskup_quenched_2021,
	title = {Quenched invariance principle for a class of random conductance models with long-range jumps},
	volume = {180},
	issn = {0178-8051,1432-2064},
	doi = {10.1007/s00440-021-01059-z},
	pages = {847--889},
	number = {3},
	journaltitle = {Probability Theory and Related Fields},
	shortjournal = {Probab. Theory Related Fields},
	author = {Biskup, Marek and Chen, Xin and Kumagai, Takashi and Wang, Jian},
	date = {2021},
	mrnumber = {4288333},
}

@article{biskup_scaling_2004,
	title = {On the scaling of the chemical distance in long-range percolation models},
	volume = {32},
	issn = {0091-1798,2168-894X},
	doi = {10.1214/009117904000000577},
	pages = {2938--2977},
	number = {4},
	journaltitle = {The Annals of Probability},
	shortjournal = {Ann. Probab.},
	author = {Biskup, Marek},
	date = {2004},
	mrnumber = {2094435},
}

@misc{jacob2024longedgeserasesubcritical,
      title={Only long edges can erase the subcritical annulus-crossing phase of weight-dependent random connection models}, 
      author={Emmanuel Jacob},
      year={2024},
      eprint={2311.04023},
      archivePrefix={arXiv},
      primaryClass={math.PR},
      url={https://arxiv.org/abs/2311.04023}, 
}

@article{hutchcroft_sharp_2022,
	title = {Sharp hierarchical upper bounds on the critical two-point function for long-range percolation on $\mathbb{Z}^d$},
	volume = {63},
	issn = {0022-2488,1089-7658},
	doi = {10.1063/5.0088450},
	pages = {Paper No. 113301, 18},
	number = {11},
	journaltitle = {Journal of Mathematical Physics},
	shortjournal = {J. Math. Phys.},
	author = {Hutchcroft, Tom},
	date = {2022},
	mrnumber = {4504407},
}

@article{coppersmith_diameter_2002,
	title = {The diameter of a long-range percolation graph},
	volume = {21},
	issn = {1042-9832,1098-2418},
	doi = {10.1002/rsa.10042},
	pages = {1--13},
	number = {1},
	journaltitle = {Random Structures \& Algorithms},
	shortjournal = {Random Structures Algorithms},
	author = {Coppersmith, Don and Gamarnik, David and Sviridenko, Maxim},
	date = {2002},
	mrnumber = {1913075},
}

\end{document}